\documentclass[11pt]{amsart}

\usepackage{amsmath,amsthm,amsfonts,amssymb,amscd,flafter,pinlabel}
\usepackage[mathscr]{eucal}
\usepackage{graphics}
\usepackage{caption, subcaption}
\usepackage[usenames,dvipsnames]{color}
\usepackage{graphicx}

\usepackage[latin1]{inputenc}

\usepackage[colorlinks=true, urlcolor=NavyBlue, linkcolor=NavyBlue, citecolor=NavyBlue, pdftitle={Braids and combinatorial knot Floer homology}, pdfauthor={Peter Lambert-Cole, Michaela Stone, David Shea Vela-Vick}, pdfsubject={Combinatorial Knot Floer Homology}, pdfkeywords={Heegaard Floer homology, 57M27, 57R58}]{hyperref}

\newtheorem{theorem}{Theorem}[section]

\newtheorem{proposition}[theorem]{Proposition}

\newtheorem{lemma}[theorem]{Lemma}

\theoremstyle{definition}
\newtheorem{definition}{Definition}[section]

\theoremstyle{definition}
\newtheorem{remark}[theorem]{Remark}

\newcommand{\R}{\ensuremath{\mathbb{R}}}
\newcommand{\Z}{\ensuremath{\mathbb{Z}}}
\newcommand{\Q}{\ensuremath{\mathbb{Q}}}

\newcommand\HFKh{{\rm {\widehat{HFK}}}}
\newcommand\HFKt{{\rm {\widetilde{HFK}}}}
\newcommand\HFKm{{\rm {HFK^{-}}}}

\newcommand\CFKh{{\rm {\widehat{CFK}}}}
\newcommand\CFKm{{\rm {CFK^{-}}}}
\newcommand\CFKt{\widetilde{\mathrm{CFK}}}

\newcommand\alphas{\mbox{\boldmath$\alpha$}}
\newcommand\betas{\mbox{\boldmath$\beta$}}

\newcommand\ws{\mathbf w}
\newcommand\zs{\mathbf z}
\newcommand\F{\mathbb F}
\newcommand\M{\mathcal M}

\newcommand\TT{\mathbb{T}}
\newcommand\SH{\mathcal{H}}

\def\x{\mathbf{x}}

\newcommand\bs{\backslash}

\def\y{\mathbf{y}}
\def\z{\mathbf{z}}
\def\w{\mathbf{w}}

\newcommand{\mr}{\mathrm}

\newcommand{\cM}{\mathcal{M}}
\newcommand{\as}{\mathbf{a}}
\newcommand{\CC}{\mathbb{C}}
\newcommand{\RR}{\mathbb{R}}
\newcommand{\del}{\partial}

\newcommand{\ZZ}{\mathbb{Z}}
\newcommand{\cW}{\mathcal{W}}
\newcommand{\cD}{\mathcal{D}}

\begin{document}

\title[{Braids and combinatorial knot Floer homology}]{Braids and combinatorial knot Floer homology}

\author[P. Lambert-Cole]{Peter Lambert-Cole}
\address{Department of Mathematics \\ Louisiana State University}
\email{plambe7@lsu.edu}
\urladdr{\href{https://www.math.lsu.edu/~plambe7/}{https://www.math.lsu.edu/\~{}plambe7}}

\author[M. Stone]{Michaela Stone}
\address{Department of Mathematics \\ Louisiana State University}
\email{mston16@lsu.edu}
\urladdr{\href{https://www.math.lsu.edu/~mston16/}{https://www.math.lsu.edu/\~{}mstone16}}

\author[D. S. Vela-Vick]{David Shea Vela-Vick}
\address{Department of Mathematics \\ Louisiana State University}
\email{shea@math.lsu.edu}
\urladdr{\href{https://www.math.lsu.edu/~shea}{https://www.math.lsu.edu/\~{}shea}}



\begin{abstract}

We present a braid-theoretic approach to combinatorially computing knot Floer homology.  To a knot or link $K$, which is braided about the standard disk open book decomposition for $(S^3,\xi_{std})$, we associate a corresponding multi-pointed nice Heegaard diagram.  We then describe an explicit algorithm for computing the associated knot Floer homology groups.  We compute explicit bounds for the computational complexity of our algorithm and demonstrate that, in many cases, it is significantly faster than the previous approach using grid diagrams.

\end{abstract}

\maketitle

\section{Introduction} 
\label{sec:introduction}

Knot Floer homology is a powerful invariant of knots and links defined independently by Ozsv\'ath and Szab\'o in \cite{OS3} and by Rasmussen in \cite{Ra}.  It is part of the general Heegaard Floer package \cite{OS1}, and has proven itself to be a tremendously useful and versatile invariant.  Knot Floer homology contains a wealth of geometric information about knots in the 3--sphere and in arbitrary 3--manifolds.  It is capable of detecting a knot's genus \cite{OS9}, determining if that knot is fibered \cite{Gh,Ni}, and also contains a powerful concordance invariant which provides lower bounds for the slice genus \cite{OS8,Ra}.  The knot Floer homology chain complex for a knot $K$ also contains information about the Heegaard Floer homology of manifolds obtained via surgeries along $K$ \cite{OS7,OS6}.

In 2006, Manolescu, Ozsv\'ath and Sarkar presented a combinatorial method for computing knot Floer homology via grid diagrams \cite{MOS}.  Their construction is based on a result of Sarkar and Wang who described a method for combinatorially computing versions of Heegaard Floer homology in general via ``nice'' Heegaard diagrams \cite{SW}.  The straightforward, combinatorial nature of Manolescu, Ozsv\'ath and Sarkar's construction has lead to many interesting and significant applications, including the discovery by Oszv\'ath, Szab\'o and Thurston \cite{OSzT} of a transverse invariant taking values in knot Floer homology.  

For knots with many crossings, computing knot Floer homology via grid diagrams is impractical.  Indeed, the computation grows factorially as a function of the grid-number (arc-index) and there exist torus knots for which the time complexity is $O(((c+2)!)^3)$, where $c$ is the crossing number.  The present goal is to present an alternative method for combinatorially computing knot Floer homology which is often faster than via grid diagrams.  Our approach goes by way of braid representations of topological knots, and is based on earlier work of Baldwin, V\'ertesi and the third author equating the various transverse invariants defined in the context of knot Floer homology \cite{BVV}.  Our main theorem is the following:

\begin{theorem}\label{thm:alg_time}
	Given a braid word $w$ representing a braid $\delta \in B_n$ in $n$--strands, the knot Floer homology $\HFKm(K(\delta))$ of the associated braid closure can be computed in $O((\frac{4^{l(w)}}{2n} + 4)^{6n} + 4^{l(w)} + n)$ time, where $l(w)$ is the length of the braid word $w$.  Moreover, there exists some positive real constant $c$, not depending on $n$ or $l(w)$, such that $\HFKm(K(\delta))$ can be computed in $O(c^{l(w)n} n^{-6n})$
\end{theorem}

To prove Theorem~\ref{thm:alg_time}, we begin by associating to any braid word $w \in B_n$ in $n$ strands a natural multi-pointed Heegaard diagram for the corresponding braid closure that we call a ``braid diagram''.  Roughly speaking, a braid diagram is a Heegaard decomposition of the knot complement, modified slightly to incorporate information about the {\it transverse} link naturally associated to the braid word $w$ (see \cite{BVV}).  The multi-pointed Heegaard diagrams one obtains via this process are described in Section~\ref{sub:braid_diagrams}.  They are not necessarily nice in the sense of Sarkar and Wang, but are nearly so -- they can contain at most $n-1$ six-sided bad regions, where $n$ is the braid index.  From here, the key step is to perform an appropriate collection of stabilizations and isotopies to connect the above bad regions to basepoint regions.  The result is a nice diagram for the associated braid closure.

There are many examples for which computing knot Floer homology via our method is significantly faster than the corresponding computation in the grid setting.  Generally speaking, this method is well-adapted to computing knot Floer homology for braid closures with large twist regions.  As a simple example, via our method, the knot Floer complex of the $(2,n)$--torus knot can be computed in polynomial time and has $(2n+1)$ generators, whereas the corresponding computation in the grid setting grows factorially in $n$ and requires $(n+2)!$ generators.  In Section~\ref{sec:example}, we compare the time required to compute knot Floer homology via our method and grid diagrams for general $(p,q)$--torus knots.

Finally, we remark that it should be possible to apply these methods to combinatorially compute the transverse braid invariant $t$, defined in \cite{BVV}.  Doing so will require a careful analysis of the isomorphisms induced on knot Floer homology by curve isotopies described in Section~\ref{sec:nice_diagram}.  This is a problem we plan to return to in a future paper.

\subsection*{Organization} 
\label{sub:organization}
In Section~\ref{sec:preliminaries}, we provide some background on knot Floer homology, grid diagrams and braid diagrams.  The algorithm and computing time required to produce a nice diagram associated to a braid is presented in Section~\ref{sec:the_algorithm}.  In Section~\ref{sec:homology}, we prove Theorem~\ref{thm:alg_time}, establishing an upper bound on the time required to compute the knot Floer homology complex via our algorithm.  Finally, in Section~\ref{sec:example}, we compare and contrast our method with the traditional grid approach by focusing on the example of torus knots. 


\subsection*{Acknowledgements} 
\label{sub:acknowledgements}

We would like to thank Samuel Connolly, Robert Lipshitz, Dylan Thurston and Ellen Weld for helpful conversations.  We would also like to extend our sincere gratitude to Louisiana State University for sponsoring the 2012 LSU Research Experience for Undergraduates.  Several of the ideas contained in this work originated as offshoots of this program.  Finally, Lambert-Cole and Vela-Vick would like to acknowledge partial support from NSF Grant DMS-1249708.


\section{Preliminaries} 
\label{sec:preliminaries}

In this section, we collect a few facts about knot Floer homology, Legendrian and transverse knots, and their invariants which will be useful in what follows.

\subsection{Knot Floer homology} 
\label{sub:hfk}

In this subsection, we review some basic definitions and results from knot Floer homology.  Here and for the remainder of the paper we work with coefficients in $\F = \Z/2$.  For a more in-depth and elementary treatment, we refer the interested reader to \cite{OS3,OS5}.

A multi-pointed Heegaard diagram for a (null-homologous) knot or link $K \subset S^3$ is a tuple $\SH = (\Sigma,\alphas,\betas,\z,\w)$ consisting of the following:
\begin{itemize}
	\item A genus $g$ Riemann surface $\Sigma$
	\item Collections of disjoint simple, closed curves $\alphas = \{ \alpha_1 \cup \dots \cup \alpha_{g+n-1} \}$ and $\betas = \{ \beta_1 \cup \dots \cup \beta_{g+n-1}\}$, each of which span $g$--dimensional subspaces of $\mr{H}_1(\Sigma;\Z)$
	\item Collections of basepoints $\z = \{z_1,\dots, z_{n} \}$ and $\w = \{w_1,\dots,w_{n}\}$ such that each component of $\Sigma \backslash (\cup \alpha_i)$ and $\Sigma \backslash (\cup \beta_i)$ contains exactly one $z_i$ and one $w_i$.
\end{itemize}
We require that the triple $(\Sigma,\alphas,\betas)$ be a Heegaard diagram for the 3--sphere, and that the knot $K \subset S^3$ is specified by the collections of basepoints $\z$ and $\w$ as follows.  First, choose oriented, embedded arcs $\gamma_1,\dots, \gamma_{n}$ in the complement $\Sigma \backslash (\cup \alpha_i)$ connecting the $z$ to $w$--basepoints.  Next, choose oriented, embedded arcs $\delta_1,\dots,\delta_{n}$ in the complement $\Sigma \backslash (\cup \beta_i)$ connecting the $w$ to $z$--basepoints.  Finally, depressing the interiors of these arcs into the $\alpha$ and $\beta$--handlebodies respectively, we require that their union $K = (\cup \gamma_i)\cup(\cup \beta_i)$ specify the knot $K$.

To each basepoint $w_i \in \w$, we associate a formal variable $U_{w_i}$.  The knot Floer complex $\CFKm(\SH)$ is then the free $\Z[U_{w_1},\dots,U_{w_{n}}]$--module generated by the intersections of the two $(g+n-1)$--dimensional tori $\TT_\alpha = \alpha_1 \cup \dots \cup \alpha_{g+n-1}$ and $\TT_\beta = \beta_1 \cup \dots \cup \beta_{g+n-1}$ inside the $(g+n-1)$--fold symmetric product $\mr{Sym}^{g+n-1}(\Sigma)$.  The differential on $\CFKm(\SH)$ is obtained as follows.  Let $\x, \y \in \TT_\alpha \cap \TT_\beta$ be a pair of generators for $\CFKm(\SH)$, $\phi \in \pi_2(\x,\y)$ a Whitney disk connecting them, and $J_t$ a generic path of almost complex structures on $\mr{Sym}^{g+n-1}(\Sigma)$.  Let $\M(\phi)$ be the moduli space of pseudo-holomorphic representatives of the disk $\phi$, and denote by $\widehat{\M}(\phi)$ its quotient by the natural $\R$--action given by translation.  The differential on $\CFKm(\SH)$ is then given by
\[
	\partial^- \x = \sum_{\y \in \TT_\alpha \cap \TT_\beta} \sum_{\substack{\phi \in \pi_2(\x,\y),\\ \mu(\phi) = 1,\\ n_\z(\phi) = 0}} \# \widehat{\M}(\phi) \cdot U_{w_1}^{n_{w_1}(\phi)} \dots U_{w_{n}}^{n_{w_{n}}(\phi)} \cdot \y,
\]
where $\mu(\phi)$ is the Maslov index of $\phi$ and $n_{w_i}(\phi)$ denotes the algebraic intersection number of the Whitney disk $\phi$ with the subvariety $\{w_i\} \times \mr{Sym}^{g+n-2}(\Sigma)$.

The knot Floer complex possesses two natural types of gradings.  The first is the Maslov (homological) grading, which is an absolute $\Q$--grading, specified up to an overall shift by the relation
\[
	M(\x) - M(\y) = \mu(\phi) - 2 \cdot \sum_i n_{w_i}(\phi),
\]
for any pair of generators $\x,\y \in \TT_\alpha \cap \TT_\beta$ and Whitney disk connecting them $\phi \in \pi_2(\x,\y)$, and the requirement that multiplication by any of the formal variables $U_{w_i}$ drop Maslov grading by 2.  The second type of grading is known as the Alexander grading.  If $K = K_1 \cup \dots \cup K_\ell$ is an $\ell$--component link, then for each $i \in \{1,\dots,\ell\}$, the Alexander grading associated to $K_i$ is an absolute $\Q$--grading, specified up to an overall shift by the relation
\[
	A_{K_i}(\x) - A_{K_i}(\y) = n_{\z_{K_i}}(\phi) - n_{\w_{K_i}}(\phi),
\]
where $\x$, $\y$ and $\phi$ are as above, $\z_{K_i}$ and $\w_{K_i}$ are the collections of basepoints corresponding to the link component $K_i$, and the requirement that multiplication by a formal variable $U_{w_j}$ drop $A_{K_i}$ by 1 if $w_j \in \z_{K_i}$ and otherwise preserve the grading. 

The ``minus'' version of knot Floer homology is the homology of the complex $(\CFKm(\SH),\partial^-)$:
\[
	\HFKm_*(K) := \mr{H}_*(\CFKm(\SH),\partial^-).
\]
When $w_i$ and $w_j$ are basepoints corresponding to the same component of the link $K$, then their associated formal variables $U_{w_i}$ and $U_{w_j}$ act identically on $\HFKm(K)$.  Choose for each component $K_i$ of the link $K$, a formal variable $U_i$ associated to some basepoint for $K_i$.  Then the knot Floer homology $\HFKm(K)$ is an invariant of the link $K \subset S^3$, which is well-defined up to graded $\F[U_1,\dots,U_\ell]$--module isomorphism.

There are two additional associated homology theories with which one commonly works.  The first is known as the ``hat'' version of knot Floer homology and is obtained as follows.  For each component $K_i \in K$, set exactly one of its associated formal variables $U_{w_i} = 0$ and denote by $\widehat{\partial}$ the associated differential on the quotient complex $(\CFKh(K),\widehat{\partial})$.  It follows that the homology
\[
	\HFKh_*(K) := \mr{H}_*(\CFKh(\SH),\widehat{\partial}),
\]
is an invariant of the link $K$ up to $\F$--module isomorphism.  Finally, it is often convenient to work with the further quotient of $(\CFKh(\SH),\widehat{\partial})$ that is obtained by setting the remaining formal variables $U_{w_j} = 0$.  The result is known as the ``tilde'' version of knot Floer homology and its complex is denoted $(\CFKt(\SH),\widetilde{\partial})$.  The associated homology
\[
	\HFKt_*(K) := \mr{H}_*(\CFKt(\SH),\widetilde{\partial})
\]
is an invariant of the link $K$ together with the number of basepoints $(n_1,\dots,n_\ell)$ in $\w$ corresponding to each component $K_1,\dots,K_\ell$ of $K$.  As $\F$--modules, we have that
\[
	\HFKt(K) \simeq \HFKh(K) \otimes V_1^{n_1-1} \otimes \dots \otimes V_\ell^{n_\ell-1},
\]
where $V_i$ is a rank 2 vector space spanned by vectors $v$ and $w$ with multi-gradings $M(v) = A_{K_j}(v) = 0$ for all $j$, and $M(w) = A_{K_i}(w) = -1$, $A_{K_j}(w) = 0$ for $j \neq i$.


\subsection{Combinatorial computations and grid diagrams} 
\label{sub:combinatorial_computations_and_grid_diagrams}

In \cite{SW}, Sarkar and Wang described a combinatorial method for computing Heegaard Floer invariants via so-called nice diagrams.

\begin{definition}\label{def:nice}
	A multi-pointed Heegaard diagram $\SH = (\Sigma,\alphas,\betas,\z,\w)$ is called {\it nice} if every region in $\Sigma \bs (\alphas \cup \betas)$ not containing a $z$--basepoint is topologically a disk with at most 4 corners.  In other words, every region in the complement of the $\alpha$ and $\beta$--curves either contains a $z$--basepoint or is a bigon or square.
\end{definition}

If a multi-pointed Heegaard diagram $\SH$ is nice, Sarkar and Wang showed that the differential on $\CFKm(\SH)$ can be computed combinatorially by counting embedded, empty rectangles and bigons connecting generators.  They further showed algorithmically how, through a sequence of handleslides and isotopies in the complement of basepoints, any multi-pointed Heegaard diagram can be transformed into one which is nice.  

If $K \subset S^3$ is a knot or link in the 3--sphere, then $K$ can be represented combinatorially via a {\it grid diagram}.  A grid diagram $G = (n;\mathbb{X},\mathbb{O})$ consists of the following:
\begin{enumerate}
	\item An $n \times n$ square planar grid,
	\item Collections $\mathbb{X} = (X_1,\dots,X_n)$ and $\mathbb{O} = (O_1,\dots,O_n)$ of $X$'s and $O$'s in the squares of the grid such that each row and each column contains exactly one $X$ and one $O$, and no square is occupied by both an $X$ and an $O$.
\end{enumerate}
To such a grid diagram $G$, one associated a knot or link in $S^3$ as follows.  First, draw (oriented) vertical line segments connecting the $X$'s and $O$'s in each column.  Next, draw (oriented) horizontal line segments connecting the $O$'s and $X$'s so that the horizontal strands pass underneath the vertical strands.  If we allow the $X$'s and $O$'s to play the roles of $z$'s and $w$'s respectively, then grid diagrams are clearly nice in the sense of Sarkar and Wang, and thus provide an avenue for combinatorially computing knot Floer homology.

In the context of Heegaard Floer theory, grid diagrams first appeared in the work of Manolescu, Ozsv\'ath and Sarkar \cite{MOS}, and Manolescu, Ozsv\'ath, Szab\'o and Thurston \cite{MOST}.  Since their initial appearance, grid diagrams have proven an essential tool for computing knot Floer homology and studying its applications to problems in topology, and in contact and symplectic geometry.


\subsection{Braids and their associated Heegaard diagrams} 
\label{sub:braid_diagrams}

Recall the well-known isomorphism between the braid group $B_n$ and the mapping class group $\mathcal{M}(D,n)$ of the disk with $n$ punctures.  

The braid group $B_n$ is the finitely presented group
\[B_n := \left\langle \sigma_1,\dots,\sigma_{n-1} \middle|
\begin{array}{r l l}
\sigma_i \sigma_{i+1} \sigma_i \, \, \,=& \sigma_{i+1} \sigma_i \sigma_{i+1} & \\
\sigma_i \sigma_j \,\, \,=& \sigma_j \sigma_i & \quad \text{if } |i-j| \geq 2
\end{array}
\right\rangle \]
This presentation of $B_n$ is the {\it Artin presentation} and the elements $\sigma_1,\dots,\sigma_{n-1}$ are {\it Artin letters}.

Let $Q = \{1,2,\dots,n\}$ be a collection of $n$ points in $\mathbb{C}$ and let $D$ be an embedded disk in $\mathbb{C}$ containing $Q$.  A self-homeomorphism $\phi$ of the pair $(D,Q)$ is an orientation-preserving homeomorphsim $\phi: D \rightarrow D$ that fixes $Q$ setwise and $\partial D$ pointwise.  The \textit{mapping class group} $\cM(D,n)$ is the group of isotopy classes of self-homeomorphisms $\phi$, with group multiplication given by composition of maps.  

A \textit{spanning arc} is a properly embedded arc $\gamma$ in the interior of $D$ such that the boundary of $\gamma$ is two disjoint points in $Q$ and the interior of $\gamma$ is disjoint from $Q$.  A \textit{half-twist} $\tau_{\gamma}$ along a spanning arc $\gamma$ is the following self-homeomorphism of the disk.  Choose a neighborhood $\nu(\gamma)$ of $\gamma$ in $D$ disjoint from the other points in $Q$ and an orientation-preserving identification $\psi$ of $\nu(\gamma)$ with the disk $B = \{|z| \leq 3\} \subset \CC$ that sends $\gamma$ to the interval $[-1,1]$.  Choose a continuous function $g: \RR \rightarrow \RR$ with $g(x) = 0$ for $x \geq 3$ and $g(x) = -1$ for $x \leq 2$.  Let $f: \CC \rightarrow \CC$ be the self-homeomorphism of $\CC$ defined by $f(z) = e^{g(|a|) \imath \pi} z$.  Then the half-twist $\tau_{\gamma}$ along $\gamma$ is defined to be $\tau_{\gamma} = \psi^{-1} \circ f \circ \psi$.  Note that a half-twist is detemined by the image of a curve that transversely intersects the spanning arc $\gamma$.  See Figure~\ref{figure:half_twist}.

\begin{figure}[!htbp]
\centering
\labellist
	\small\hair 2pt
	\pinlabel $\gamma$ at 80 160
\endlabellist
\includegraphics[width=.5\textwidth]{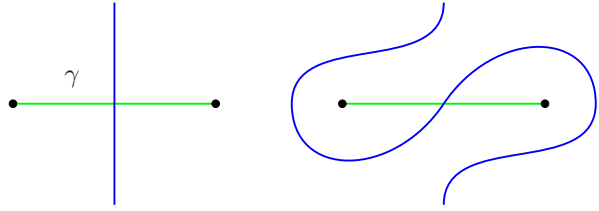}
\caption{A half-twist $\tau_{\gamma}$ along the spanning arc $\gamma$}
\label{figure:half_twist}
\end{figure}

Let $\gamma_i = [i,i+1] \times \{0\}$ denote the spanning arc connecting the $i^{\text{th}}$ and $(i+1)^{\text{st}}$ basepoint in $Q$.  Then there is a group isomorphism 
\[B_n \simeq \cM (D,n)\]
given by the map $\sigma_i \leftrightarrow \tau_{\gamma_i}$ that sends the $i^{\text{th}}$ Artin generator to the half twist along the $i^{\text{th}}$ spanning arc.

Now, let $L$ be a link in $S^3$, $\delta \in B_n$ a $n$--strand braid whose closure is $L$, and $\phi \in \cM(D,n)$ the mapping class corresponding to $\delta$.

In the disk $D$, label the basepoints $z_j = (j,0)$ and choose a {\it basis} $\as$ for $(D, \zs)$, that is, a collection $\{a_1,\dots,a_{n-1} \}$ of properly embedded arcs in $D$ such that each component of $D \setminus \as$ contains exactly one $z$ basepoint.  Index $\as$ so that $z_j$ lies in the component bounded by $\del D, a_{j-1}, a_j$.  Let $\mathbf{b}$ be a second basis where each $b_j$ is obtained by pushing off $a_j$ along the orientation of $\del D$ and isotoping so that $a_j$ and $b_j$ intersect transversely in a unique point.

 Let $D'$ denote a second copy of this disk, with identical basepoints $\ws = \{w_j\}$ and first basis $\as' = \{a'_j\}$.  Let $\mathbf{b}'$ be a second basis defined by setting $b'_j = \phi(b_j)$.  We can assume, after possibly a perturbation of $\phi$, that each pair $a'_i$ and $b'_j$ intersect transversely.

\begin{figure}[!htbp]
\centering
\labellist
	\small\hair 2pt
	\pinlabel $z_1$ at 34 145
	\pinlabel $a_1$ at 70 50
	\pinlabel $b_1$ at 55 215
	\pinlabel $z_2$ at 114 145
	\pinlabel $a_2$ at 132 50
	\pinlabel $b_2$ at 119 215
	\pinlabel $z_3$ at 179 145
	\pinlabel $a_3$ at 197 50
	\pinlabel $b_3$ at 181 215
	\pinlabel $z_4$ at 242 145
	\pinlabel $z_{n-2}$ at 340 145
	\pinlabel $a_{n-2}$ at 346 50
	\pinlabel $b_{n-2}$ at 333 215
	\pinlabel $z_{n-1}$ at 407 145
	\pinlabel $a_{n-1}$ at 410 50
	\pinlabel $b_{n-1}$ at 399 215
	\pinlabel $z_n$ at 483 145
	\pinlabel $\dots$ at 285 130
\endlabellist
\includegraphics[width=.6\textwidth]{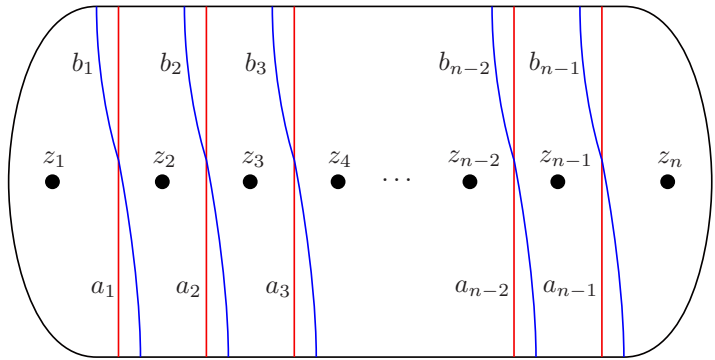}
\caption{The disk $D$ with bases $\as, \mathbf{b}$}
\label{figure:top_disk}
\end{figure}

\begin{figure}[!htbp]
\centering
\includegraphics[width=.5\textwidth]{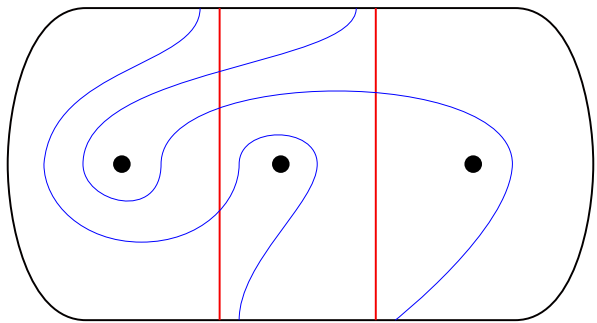}
\caption{The disk $-D'$ for $n=3$ and $\delta = \sigma_2 \sigma_1$}
\label{figure:bottom_disk}
\end{figure}

Let $\Sigma = D \cup - D'$ and for $i = 1,\dots,n$, let $\alpha_i = a_i \cup a_i'$ and $\beta_i = b_i \cup b_i'$.

\begin{lemma} 
\label{lemma:braid_Heegaard_diagram}
For an oriented link $L$, let $\delta \in B_n$ be a braid whose closure is $L$ and let $\phi \in \mathcal{M}(D,n)$ be the mapping class identified with $\delta$.  Then $\SH_{\phi} = (\Sigma, \alphas, \betas, \zs, \ws)$ is a multi-pointed Heegaard diagram for $-L$.\footnote{That the Heegaard diagram $\SH_{\phi}$ specifies the oriented link $-L$ as opposed to $L$ arises as a consequence of our choice to place the $z$-basepoints on $D$.  This was done to allow for future applications to the study of transverse knot theory, a topic we hope to return to later.} 
\end{lemma}

\begin{remark}
\label{rem:dual_diagram}
It is often useful to consider the ``dual'' diagram obtained by using $\phi^{-1}$ to reparametrize $\Sigma$.  Diagrammatically, this is accomplished by instead choosing bases $\as' = \{a'_i = \phi^{-1}(a_i)\}$ and $\mathbf{b}' = \{b'_i = b_i\}$ for $D'$ before constructing $\SH_{\phi}$.  In general, modifying the $\mathbf{b}'$ by some $\phi$ is equivalent to instead modifying the $\as'$ by $\phi^{-1}$.
\end{remark}

We call a multi-pointed Heegaard diagram $\SH_{\phi}$ obtained in the above fashion a {\it braid diagram}.  In the context of Heegaard Floer theory, such diagrams first appeared in the work of Baldwin, V\'ertesi and the third author in their work establishing an equivalence relating certain transverse invariants in knot Floer homology \cite{BVV}.  There, it is shown that the distinguished collection of intersections $\x$ lying on the disk $D \subset \Sigma$ together define a cycle in $\CFKm(\Sigma, \betas, \alphas, \z, \w)$ whose associated homology class $[\x]$ is an invariant of the transverse knot of link type of the associated braid.



\section{The algorithm} 
\label{sec:the_algorithm}

In order to algorithmically compute the Knot Floer Homology of an oriented link $L \subset S^3$, we need to construct a multi-pointed Heegaard diagram for $L$ that is nice, in the terminology of Sarkar-Wang.

Starting with a braid $B \in B_n$ whose closure is $-L$ and a braid word $w$ in the Artin generators representing $B$, our approach is to
\begin{enumerate}
\item Use the braid word $w$ to construct a specific, ``efficient'' self-homeomorphism of the disk with $n$ punctures whose mapping class corresponds to $B$
\item Use this self-homeomorphism to construct an efficient multi-pointed Heegaard diagram for the original link $L$, as described in Subsection \ref{sub:braid_diagrams}
\item Apply an appropriately modified version of the Sarkar-Wang algorithm to make the diagram nice.
\end{enumerate}

\subsection{Nice diagram} 
\label{sec:nice_diagram}
In this first subsection, we will describe the Heegaard diagram abstractly before describing how to algorithmically obtain the diagram in Subsection \ref{subsec:Algorithm}.

Let $A_i$ and $B_i$ denote the components of $\Sigma \setminus \alphas$ and $\Sigma \setminus \betas$, respectively, containing the $i^{\text{th}}$ basepoint $z_i$.  Thus, $\partial A_i = \alpha_i - \alpha_{i-1}$ and $\partial B_i = \beta_i - \beta_{i-1}$.

A \textit{trivial bigon} is a region in the disk $D'$, disjoint from any basepoint $w_i$, whose boundary splits into two connected components, one a segment of an $\alpha$--curve and the other a segment of a $\beta$--curve.  Such a bigon can be eliminated by an isotopy of the $\alpha,\beta$ curves as it contains no basepoint to obstruct this isotopy.  A monodromy $\phi$ that produces no trivial bigons in the Heegaard diagram is called {\it efficient}. 

Choosing an efficient $\phi$ ensures that the diagram $\SH_{\phi}$ is already very close to being nice.

\begin{lemma}
\label{lemma:low_badness}
Every elementary region of $\SH_{\phi}$ not containing some $z_j$ is a $2k$--gon.  Furthermore, suppose that $\phi$ is efficient and let $i = 1,\dots,n$.  Then among the elementary regions of $A_i$, there is at most one 6--gon and no $2k$--gons for $k \geq 4$.  Similarly, among the elementary regions of $B_i$, there is at most one 6--gon and no $2k$--gons for $k \geq 4$.
\end{lemma}

\begin{proof}

By construction, $A_1, B_1, A_n, B_n$ are topological disks and $A_i, B_i$ are annuli for $i = 2,\dots,n-1$.  Also, no $\alpha_i$ is contained within a single $B_j$ and no $\beta_j$ is contained within a single $A_i$.

Let $R$ be an elementary region, contained in some $A_i$ and in some $B_j$.  Thus, segments of the boundary of $R$ can only form parts of four curves: $\alpha_{i-1},\alpha_i,\beta_{j-1},\beta_j$.  Moreover, each boundary component must have $\alpha$ and $\beta$ edges.

Note that each $\alpha$ and $\beta$ can appear as an edge in at most 1 boundary component. Since $\Sigma = S^2$, there is a 1-1 correspondence between the boundary components of $R$ and connected components of $\Sigma \setminus R$.  Therefore, a path in $\Sigma$ between boundary components of $R$ must be contained in $R$.  Because each $\alpha,\beta$ is closed, if it forms part of two boundary components, it must lie entirely within $R$.  But since $R$ is elementary, that curve is actually the totality of a single boundary component.

Thus, $R$ has either 1 or 2 boundary components and can be either a topological disk or an annulus.

Suppose that $R$ is an annulus and let $\epsilon$ be an properly embedded arc in $A_i$ connecting $\alpha_{i-1}$ to $\alpha_i$.  Then either $\epsilon$ is contained in $R$ or $\epsilon$ intersects the $\beta$ edges of $\del R$ at least once.  There is clearly an arc in the disk $D$ connecting $\alpha_{i-1}$ to $\alpha_i$ that passes through $z_i$ and is disjoint from every $\beta$ curve.  Hence $R$ must contain $z_i$.

Now suppose $R$ is a $2k$--gon.  If some $\beta$ edge of $\del R$ connects an $\alpha$ curve to itself, the union of this edge and some edge of that $\alpha$ forms a bigon.  Since $\phi$ is efficient, that bigon must contain a basepoint and that basepoint must be $w_i$.  Since $w_i$ is unique, for each pair of bigons formed in this way, one must either contain the other.  If there are multiple such bigons, then $R$ must be contained in one and $R$ is a square.

If $R$ is not a bigon or such a square, then exactly two $\beta$ edges of $\del R$ connect $\alpha_{i-1}$ and $\alpha_i$.  There can be at most one more $\beta$ edge connecting some $\alpha$ to itself, so $R$ is either a square or a 6--gon.

Finally, suppose that there are two elementary 6--gons $R_1,R_2$ within $A_i$.  Then one, say $R_2$, must be contained in the nontrivial bigon associated to $R_1$.  However, this is a contradiction since $R_2$ must therefore be a square.

The exact same argument with $A_i$ replaced by $B_j$ and the $\alpha$'s and $\beta$'s exchanged proves the statement for $B_j$. 
\end{proof}

\begin{remark}
\label{rem:max_badness}
In the terminology of Sarkar-Wang, this implies that the maximal badness of any region is 1 and that there are at most $n-2$ bad regions.
\end{remark}

Now, following the approach of \cite{Hales}, we use the following {\it stabilization trick} to make the Heegaard diagram nice.

\begin{definition} Let $\SH_{\phi}$ be a Heegaard diagram with an empty, elementary 6--gon $R \subset A_i$.  The {\it stabilization trick} consists of 
\begin{enumerate}
\item Stabilize $\SH_{\phi}$ by attaching a 1--handle to $\Sigma$ with the attaching sphere given by two points, one in $R$ and the other in a $z$--basepointed region
\item Apply the Sarkar-Wang algorithm
\end{enumerate}

More specifically, choose a point $x$ in $R$ near some $\alpha$ boundary edge and a pushoff of that $\alpha$ curve containing $x$.  Let $\widehat{\alpha_i}$ be an oriented subsegment of this pushoff disjoint from the bigon containing the basepoint $w_i$ and whose boundary is $x \cup y$ for some point $y$ in the region containing $z_i$.  Then attach a 1--handle to $\Sigma$ with attaching sphere $x \cup y$ and extend $\widehat{\alpha_i}$ across the 1--handle to a closed curve.  Choose $\widehat{\beta_i}$ to be a meridian or belt-sphere of this 1--handle.  It is isotopic to a pushoff of the boundary of the 6--gon.

Finally, perform ``finger moves'' by pushing $\widehat{\beta_i}$ across the three $\alpha$ edges and continue until reaching basepointed regions.  This is easiest to see in the ``dual'' diagram in Figure~\ref{figure:stabilizedB}.

\begin{figure}
\centering
\begin{subfigure}{.45\textwidth}
	\centering
	\labellist
		\small\hair 2pt
		\pinlabel $\widehat{\alpha_i}$ at 207 320
		\pinlabel $\widehat{\beta_i}$ at 255 270
		\pinlabel $z_i$ at 92 517
		\pinlabel $w_i$ at 50 264
		\pinlabel $\alpha_i$ at -20 130
		\pinlabel $\alpha_{i+1}$ at 312 130
		\pinlabel $\dots$ at 140 130
		\pinlabel $\dots$ at 120 400
		\pinlabel $\dots$ at 138 265
	\endlabellist
	\includegraphics[width=.5\linewidth]{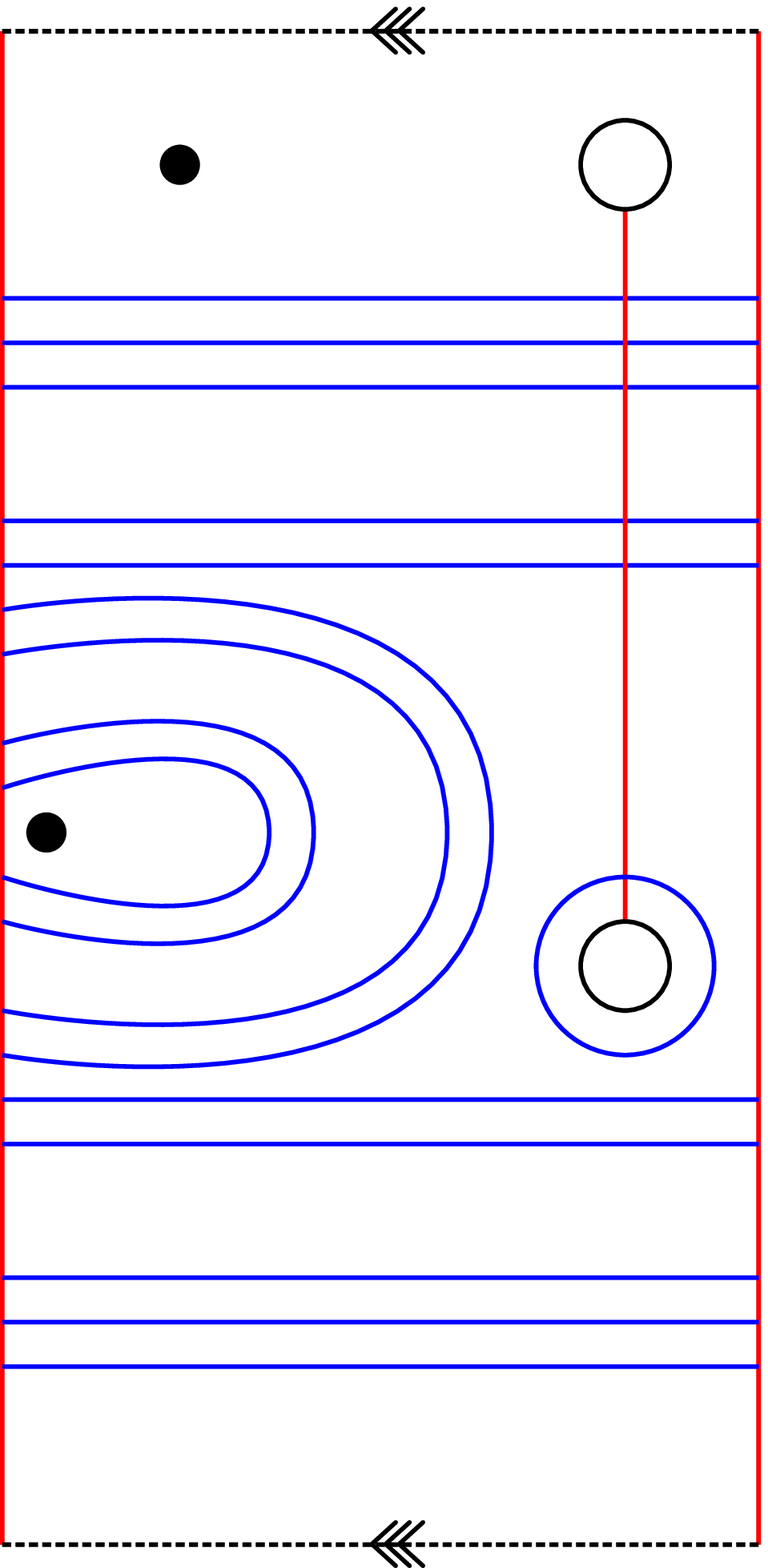}
	\caption{The annulus $A_i$ after stabilization and before finger moves}
	\label{figure:stabilizedA}
\end{subfigure}
\begin{subfigure}{.45\textwidth}
	\centering
	\labellist
		\small\hair 2pt
		\pinlabel $\widehat{\beta_i}$ at 230 325
		\pinlabel $\widehat{\alpha_1}$ at 300 220
		\pinlabel $z_j$ at 92 517
		\pinlabel $w$ at 50 264
		\pinlabel $\beta_j$ at -20 130
		\pinlabel $\beta_{j+1}$ at 310 130
		\pinlabel $\dots$ at 140 130
		\pinlabel $\dots$ at 140 400
		\pinlabel $\dots$ at 138 265
	\endlabellist
	\includegraphics[width=.5\linewidth]{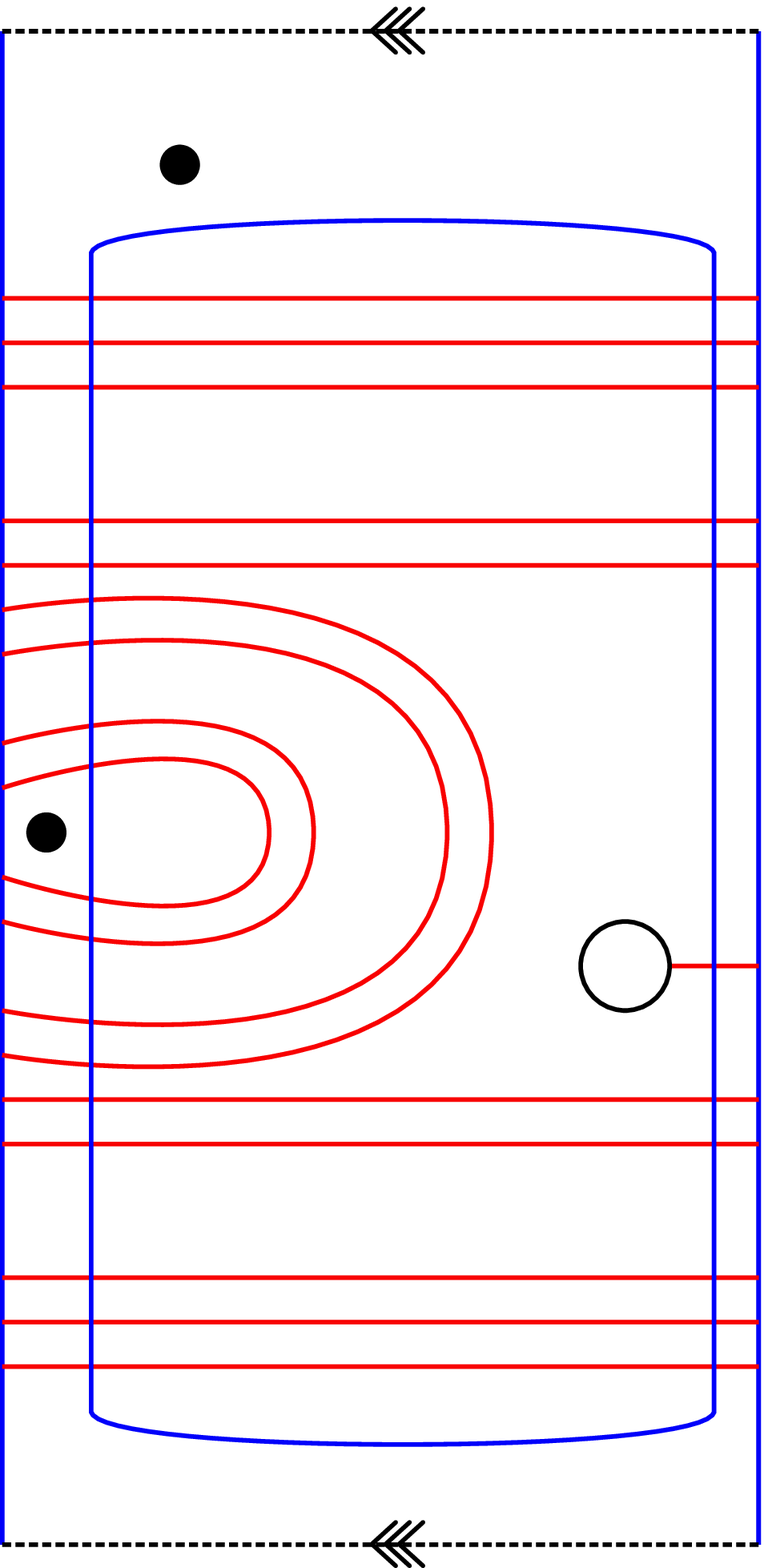}
	\caption{The ``dual'' diagram of $B_j$ after finger moves}
	\label{figure:stabilizedB}
\end{subfigure}
\caption{The Stabilization Trick}
\label{figure:stabilization}
\end{figure}

\end{definition}

Suppose that $\SH_{\phi}$ has $g$ elementary 6--gons.  Let $\SH_{\phi}^{\text{nice}} = (\Sigma_g,\alphas \cup \widehat{\alphas}, \beta \cup \widehat{\betas}, \zs, \ws)$ be the Heegaard diagram obtained by applying the stabilization trick to each 6--gon.

\begin{proposition}
\label{prop:diagram_nice}
The Heegaard diagram $\SH_{\phi}^{\text{nice}}$ is nice.
\end{proposition}

\begin{proof}
It is clear that all elementary regions in each $B_i$ are either bigons or 4--gons, except possibly for a unique 6--gon.  However, via the attached 1--handle, this 6--gon is identified with the elementary region containing the basepoint $z_i$.

Furthermore, note that unless $w_i$ and $z_i$ lie in the same elementary region, each region containing a $w_i$ basepoint is either a bigon or a square.
\end{proof}


\subsection{Algorithm} 
\label{subsec:Algorithm}

In this subsection, we will explicitly describe the nice Heegaard diagram and compute the time required to obtain the diagram.

A {\it braid word} $w$ for a braid $\delta \in B_n$ is a string $\sigma_{i_1}^{s_1} \sigma_{i_2}^{s_2} \ldots \sigma_{i_k}^{s_{k}}$, where $\sigma_{i_j} \in \{ \sigma_1,\dots,\sigma_{n-1} \}$ and $s_{j} \in \{\pm 1\}$, of Artin letters whose product equals $\delta$ in $B_n$.  The integer $k = l(w)$ is called the {\it length} of the word $w$.

Throughout this section, the algorithm depends upon the explicit word $w$ chosen for a braid $\delta$, although the final Heegaard diagram does not.  For this reason, we will use $\SH_w$ to denote Heegaard diagram.

In addition to the setup and choice of bases $\as, \mathbf{b}$ from Subsection \ref{sub:braid_diagrams}, we choose specific spanning arcs $\gamma_i$ for $i = 0,\dots,n$.  Recall that to identify $B_n$ and $\cM(D,n)$, we chose spanning arcs $\gamma_i = [i,i+1]$ connecting the basepoints. We can assume that the boundary $\partial D$ passes through the points $(0,0)$ and $(n+1,0)$, and let $\gamma_0 = [0,1]$ and $\gamma_n = [n,n+1]$ be arcs that connect the $1^{\text{st}}$ and $(n)^{\text{th}}$ basepoints to the boundary of the disk.  Furthermore, we can choose $\as$ such that $\alpha_i \cap \gamma_j$ is empty if $i \neq j$ and is exactly one point if $i = j$.  Finally, we assume that there are no triple intersection points between $\alpha,\beta,\gamma$ curves.

The spanning arc $\gamma_i$ cuts $a'_i$ into two components, $a_i^+$ and $a_i^-$, which are contained in the upper-half and lower-half planes, respectively, and $a_i$ cuts the spanning arc $\gamma_i$ into two components $\gamma_i^{\pm}$.  Let $A_i^{\pm}$ denote the number of intersections $|a_i^{\pm} \cap \betas|$ between $\beta$ curves and the components of $a'_i \setminus \gamma_i$ and define $A_i = A_i^+ + A_i^-$.  Similarly, let $\Gamma_i^{\pm}$ denote the number of intersections $|\Gamma_i^{\pm} \cap \betas|$ between $\beta$ curves and the components of $\gamma_i \setminus \alpha_i$ and define $\Gamma_i = \Gamma_i^+ + \Gamma_i^-$.  

\begin{definition}
The \textit{complexity} of a multi-pointed Heegaard diagram $\SH$ is a tuple $(n,v)$, where $n-1$ is the number of $\alpha$ curves and $v = |\alphas \cap \betas|$ is the total number of intersections between $\alpha$ and $\beta$ curves.  
\end{definition}

Finally, we need to set a unit of time.  A Heegaard diagram $\SH$ can be represented combinatorially by using $\alphas, \betas, \mathbb{\gamma}, \ws$ to determine a handle decomposition of $\Sigma$.  Specifically, there is a unique 0-handle for each intersection point between some pair of $\alpha,\beta$ or $\gamma$ curves and a unique 0-handle for each $w$-basepoint; a unique 1-handle for each segment of an $\alpha, \beta$ or $\gamma$ curve between intersection points or $w$-basepoint; and a unique 2-handle for each elementary region.  Each intersection point is 4-valent, so introducing or eliminating an intersection point requires a constant number of modifications of the handle decomposition.  Thus, we declare that adding or removing a single intersection point between a pair of $\alpha$ and $\beta$ curves or a pair of $\gamma$ and $\beta$ curves takes $O(1)$ time.

Given a braid word $\sigma_{i_1}^{s_1} \dots \sigma_{i_{l(w)}}^{s_{l(w)}}$ for a braid $\delta$, it is straightforward to build the self-homeomorphism $\phi$ corresponding to $\delta$ by starting from a diagram $\SH_1$ of the identity map and successively applying the half-twists $\tau_{i_j}^{s_j}$.  After each step, we will also immediately remove any trivial bigons formed by pairs of $\alpha$ and $\beta$ curves or pairs of $\gamma$ and $\beta$ curves.

{\bf Base case}.  The mapping class corresponding to the trivial braid $1$  is the identity $id$, so $D'$ is just a second copy of $D$ (see Figure~\ref{figure:top_disk}).  The diagram $\SH_1$ has complexity $(n, 2(n-1))$ and takes $O(n)$ time to construct.  We choose it so that $A_i^+ = \gamma_i^- = 1$ and $A_i^- = \gamma_i^+ = 0$ for all $i = 1,\dots n-1$.

{\bf Inductive step}.  Figure~\ref{figure:Heegaard_Half_Twist} describes the effect of the half-twist $\tau_i$ on the lower-half disk $-D'$ of $\SH_w$.  The half-twist is the identity outside of a neighborhood of the spanning arc $\gamma_i$ and we can assume that in this neighborhood, there are no $\alpha,\beta$ intersections.  For the following discussion, we restrict attention to the positive half-twist as the case of the inverse half-twist is identical up to mirror image.

\begin{figure}[!htbp]
\centering
\begin{subfigure}{.45\textwidth}
	\centering
	\includegraphics[width=.6\linewidth]{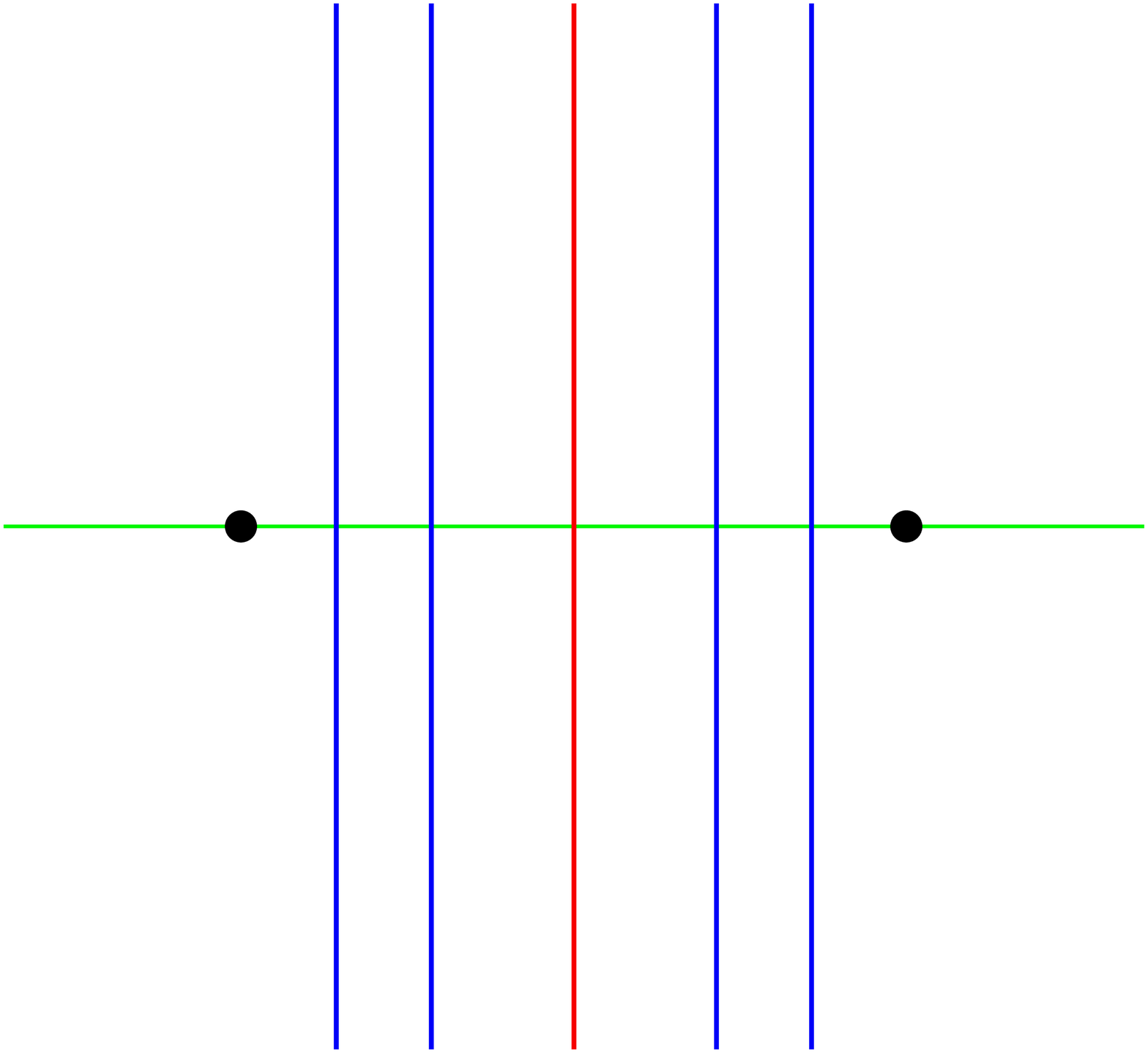}
	\caption{A neighborhood of $\gamma_i$ before the half-twist}
	\label{figure:Heegaard_Pre_Half_Twist}
\end{subfigure}
\begin{subfigure}{.45\textwidth}
	\centering
	\includegraphics[width=.6\linewidth]{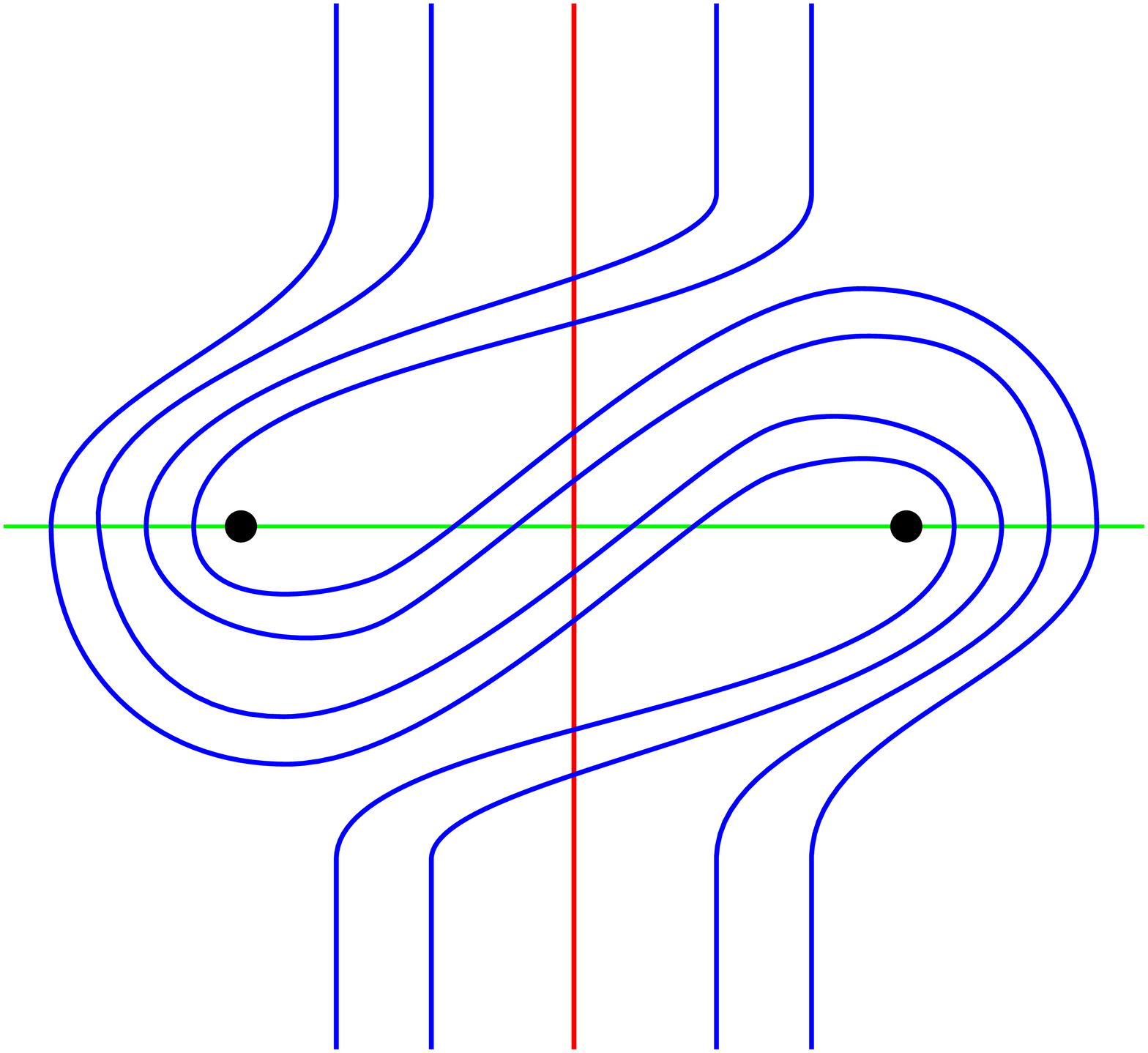}
	\caption{A neighborhood of $\gamma_i$ after the half-twist}
	\label{figure:Heegaard_Post_Half_Twist}
\end{subfigure}
\caption{A half-twist $\gamma_i$}
\label{figure:Heegaard_Half_Twist}
\end{figure}

Consider what happens to a $\beta$ strand that intersects $\gamma_i^-$ under the half-twist.  It now intersects first $\gamma_{i-1}^+$, then $a_i^-$, then $\gamma_i^+$, then $\gamma_{i+1}^-$, then finally $a_i^-$ again before leaving $\nu(\gamma_i)$.  On the other side, a $\beta$ strand that intersected $\gamma_i^+$ now intersects first $a_i^+$, then $\gamma_{i-1}^+$, then $\gamma_i^-$, then $a_i^+$, then finally $\gamma_{i+1}^-$ before leaving $\nu(\gamma_i)$.  

In order to see how new trivial bigons can be created, consider instead the ``dual'' diagram of a half-twist, in which we apply $\tau_i^{-1}$ to the $\alpha,\gamma$ curves instead in Figure~\ref{figure:new_bigons}.  New trivial bigons appear if prior to applying the half-twist, a $\beta$ strand, after it leaves $\nu(\gamma_i)$, continues counterclockwise and intersects the next $\alpha$ and/or $\gamma$ segments.

\begin{figure}[!htbp]
\centering
\includegraphics[width=.3\textwidth]{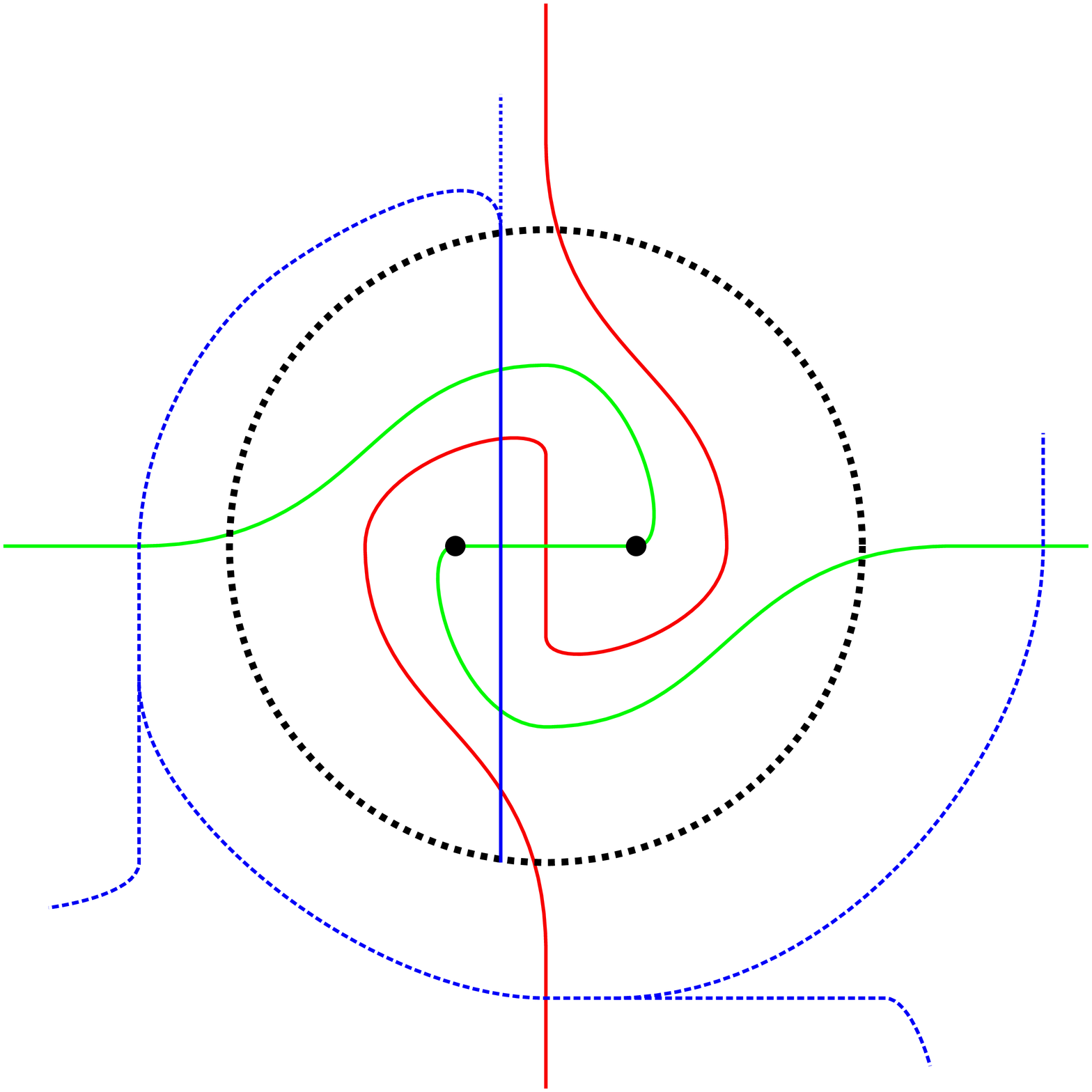}
\caption{The possible trajectories of a $\beta$ arc as it leaves $\nu(\gamma_i)$}
\label{figure:new_bigons}
\end{figure}

For example, a $\beta$ curve leaving $\nu(\gamma_i)$ at the top-left, between $\gamma_{i-1}$ and $a_i^+$, can next hit (in counter-clockwise order) $a_i^+, \del D', a_{i-1}^+, \gamma_{i-1}^+$ and if it crosses the spanning arc $\gamma_{i-1}$, it can next hit $a_{i-1}^-, \del D', a_i^-$.  (We can assume, by choosing $\nu$ small, that this $\beta$ segment does not enter $\nu(\gamma_i)$ again).

For each original $\beta$ arc in $\nu(\gamma_i)$, the half-twist $\tau_{\gamma_i}$ can introduce at most two trivial bigons, one with $\alpha_i$ and one with $\gamma_{i \pm 1}$, in each direction it leaves $\nu(\gamma_i)$ for a total of at most four new trivial bigons.

For $w = w' * \sigma_i^*$, let $\SH_{w' \sigma_i^*}$ denote the Heegaard diagram obtained by applying the half-twist $\tau_{\gamma_i}^*$ and then removing any trivial bigons.

We can now describe the complexity of the diagram and the time to obtain the diagram.

\begin{lemma}
\label{lemma:complexity_runtime}
Suppose that the complexity of $\SH_w$ is $(n,v)$.  Then the complexity of $\SH_{w \sigma_i^*}$ is at most $(n,v + 12( \Gamma_i))$ and it can be obtained from $\SH_w$ in $O(|\Gamma_i|)$.  The complexity of $\SH_w$ is at most $(n,4^{l(w)}+2n))$ and it can be obtained in $O(4^{l(w)} + n)$.
\end{lemma}

\begin{proof}
It's clear that performing this half twist introduces $2\Gamma_i$ new intersections between $\alpha_i$ and the $\beta$ curves.  Moreover, it introduces $\Gamma_i$ new intersections between $\gamma_{i-1}^+, \gamma_{i+1}^-$ and the $\beta$ curves.  Removing the trivial bigons requires eliminating 2 vertices for at most $4 \Gamma_i$ new trivial bigons.

The first pair of facts follow since at most $12 \Gamma_i$ new vertices are introduced or eliminated, each of which takes $O(1)$ times.  The second pair follows since at each step, $\Gamma_{i-1},\Gamma_{i+1}$ each increase by $\Gamma_i < \text{max}_j \Gamma_j$.  The maximum among the $\Gamma_j$ can at most double each iteration.
\end{proof}

Finally, we need to make the diagram nice.

\begin{proposition}
\label{prop:diagram_complexity}
The nice diagram $\SH_w^{nice}$ has complexity at most $(2n-2, 4^{l(w)+1} + 8n)$ and can be obtained in $O(4^{l(w)})$.
\end{proposition}

\begin{proof}
Attaching a handle in the $i^{\text{th}}$ annulus introduces at most $A_i$ intersections between $\widehat{\alpha_i}$ and the $\beta$ curves and introduces at most $B_j + B_{j-1}$ intersections between $\widehat{\beta_j}$ and the $\alpha$ curves.  In other words, we introduce at most 3 new intersection points for each original intersection point.  Since the new number of intersection points is a constant multiple of the old number of intersection points, the time complexity is the same as in Lemma~\ref{lemma:complexity_runtime}.
\end{proof}

\begin{remark}
This bound on diagram complexity may overestimate the number of new intersections of $\beta$ curves and spanning arc.  It should be possible to choose a lower base in the exponential bound.  However, the number of vertices can grow exponentially in the braid word length.  For instance, the complexity of the diagram for the pseudo-Anosov braid $(\sigma_2^{-1} \sigma_1)^k$ grows exponentially in $k$.
\end{remark}


\section{Computing homology} 
\label{sec:homology}

In this section, we will describe how to compute knot Floer homology from a nice, multi-pointed Heegaard diagram.  The approach here is fairly straightforward and there are many techniques to improve the speed of computation.  However, we are only interested in establishing the basic qualitative properties of our approach.  The discussion draws from \cite{Hales} but is adapted to our particular Heegaard diagram.

Throughout this section, we consider Heegaard diagrams $\SH_w^{nice}$ of complexity $(n+1,v)$.

First, we establish a bound on the size of the chain complex.  Recall that the knot Floer complexes are generated by the intersection points $\TT_{\alpha} \cap \TT_{\beta}$.

\begin{lemma}
\label{lemma:generator_bound}
Let $\SH$ be a Heegaard diagram of complexity $(n+1,v)$.  Then the number of generators of the knot Floer complex is at most $\left(\frac{v}{n}\right)^n$.
\end{lemma}

\begin{proof}
The number of generators is the permanent of the matrix $M$ whose $(i,j)^{\text{th}}$ entry is the number of intersection points between $\alpha_i$ and $\beta_j$.  It is an easy exercise to prove that, restricted to the level set $\sum_{i,j} M_{i,j} = v$,  the maximum of the permament function is $\frac{v}{n} * \text{Id}_n$ and the minimum is $\frac{v}{n^2} * [1]$, where $\text{Id}_n$ is the $n \times n$ identity matrix and $[1]$ is the $n \times n$ matrix with 1's in every entry.
\end{proof}

In order to compute the Alexander and Maslov gradings of the generators and to identify the holomorphic disks in the differential, we need to find chains in the Heegaard diagram connecting pairs of generators.  To accomplish this, we will use a linear map associated to the Heegaard diagram as follows.

Let $R$ denote the $\mathbb{R}$--vector space generated by elementary regions $r \in \Sigma \setminus (\alphas \cup \betas)$.  Let $C$ denote the $\mathbb{R}$--vector space generated by the $\alpha$ curves $\{\alpha_1,\dots, \alpha_{n-1+g}\}$ and all but one of the $\beta$ curves $\{\beta_1, \dots, \beta_{n-2+g}\}$.  Let $V$ denote the $\mathbb{R}$--vector space generated by the vertices of the Heegaard diagram. These spaces each have an inner product $\langle \cdot , \cdot \rangle$ defined by choosing the above bases to be orthonormal.

 Let $\rho = \text{dim } R$ be the number of regions, $v = \text{dim } V$ the number of vertices, and $e$ the number of edges in the Heegaard diagram.  Then 
\[2 - 2g = \rho - e + v\]
and since each vertex is 4--valent, the identity $2e = 4v$ implies that
\begin{equation}
\label{eq:Euler}
4 - 4g = 2 \rho -3v
\end{equation}

Let $\mathcal{D}$ be the linear map from $R \oplus C \rightarrow V$ that assigns to each region $r$ the signed sum of vertices incident to $r$ and to each $\alpha$ or $\beta$ curve to the sum of vertices along it.  The sign of a vertex $v$ is positive if, traveling along the boundary of $r$ in the direction induced by the orientation, one jumps from a $\beta$ curve to an $\alpha$ curve at $v$, and the sign is negative otherwise.  

By abuse of notation, let $A_i$ and $B_i$ denote the sum in $R$ of elementary regions constituting the connected components $A_i$ of $\Sigma \setminus \alphas$ and $B_i$ of $\Sigma \setminus \betas$, respectively, containing the basepoint $z_i$.  

A domain in $\Sigma$ is \textit{periodic} if its boundary is the union $\coprod_{i} k_i \alpha_i  \coprod_i l_i \beta_i$ of some number of $\alpha$ and $\beta$ curves.  The domains $A_i,B_j$ and sums of these domains are periodic.

\begin{lemma}
\label{lemma:chain_map_surjective}
$\mathcal{D}$ is surjective and the kernel of $\mathcal{D}$ is spanned by $A_1, \dots, A_n$ and $B_1, \dots, B_{n-1}$.
\end{lemma} 

\begin{proof}

The statement about the kernel follows from two facts.  First, $\mathcal{D}|_C$ is clearly injective and the image $\mathcal{D}(R)$ is perpendicular to the image $\mathcal{D}(C)$ with respect to the inner product since the signs of corners of a region $r$ formed by some $\alpha$ or $\beta$ cancel in pairs.  Therefore, the kernel of $\mathcal{D}$ lies in $R$.  

Secondly, the kernel of $\mathcal{D}$ in $R$ exactly corresponds to periodic domains.  It follows that since the ambient manifold is $S^3$, the periodic domains are generated by the domains $A_i,B_j$ subject to the relation $B_n = \sum_{i=1}^n A_i - \sum_{j=1}^{n-1} B_j$.

Since the nullity of $\mathcal{D}$ is $2n-1$ and the identity $4g -4 = \rho -3v$ in Equation \ref{eq:Euler} holds, this implies that $\mathcal{D}$ is surjective.
\end{proof}

 Let $\mathcal{W}$ denote the left inverse of $\mathcal{D}$.  Then $\mathcal{W}$ maps $V$ bijectively onto $\text{ker}(\mathcal{D})^{\perp}$ and $\mathcal{WD}$ is the identity on $\text{ker}(\mathcal{D})^{\perp}$.

Each generator $\x$ of the knot Floer chain complex can be identified with a vector $x = x_1 + \dots + x_{n-1+g}$ in $V$.  For a pair of generators, the map $\mathcal{W}$ thus determines a fractional chain on the surface of the Heegaard diagram connecting the two generators.

Since $\mathcal{D}$ is not injective, there is some ambiguity to recovering Whitney disks using the map $\mathcal{W}$.  However, it follows from the proof of Lemma~\ref{lemma:chain_map_surjective} that the only ambiguity is if $\x,\y$ contain distinct vertices $x',y'$ in the intersection of the same pair $\alpha_i,\beta_j$ of curves.  Then the chain $P_{i,j} = \sum_{k \leq i} A_k - \sum_{k \leq j} B_k$ is the difference of two distinct chains connecting $x'$ to $y'$.  The inverse map $\mathcal{W}$ will count each of these chains with coefficient $\frac{1}{2}$ and therefore the coefficients in $\mathcal{W}(x-y)$ lie in $\frac{1}{2} \ZZ$.  To remove this ambiguity, we can consider the chains $\mathcal{W}(x-y) \pm \frac{1}{2} P_{i,j}$.  

If the $C$ components of $\mathcal{W}(x-y)$ are 0 and the $R$ components are integers, then $\phi = \mathcal{W}(x-y)$ is a 2--chain and represents a Whitney disk in $\pi_2(\x,\y)$.  In particular, if the $R$ components of the vector $\mathcal{W}(x-y)$ are either 1 or 0 and the $C$ components are 0, then there is an embedded Whitney disk connecting $\x$ to $\y$. If we correct $\mathcal{W}(x-y)$ by adding the appropiate periodic domains $\pm \frac{1}{2} P_{i,j}$ for all such pairs $x',y'$, the collection gives exactly the Whitney disks connecting $\x$ to $\y$.  Let $\pi_2^0(\x,\y)$ denote the collection of all such domains and note that there are at most $2n$ such domains.

\textbf{Alexander grading}.  The relative Alexander grading of a pair $\x,\y$ is determined by the formula:
\[A(\x) - A(\y) = n_z(\phi) - n_w(\phi)\]
Let $Z,W$ be the vectors in $R$ that are the sums of the basis vectors corresponding to regions with a $z$ or $w$ basepoint, respectively.  Then, the relative Alexander grading is given by
\[A(\x) - A(\y) = \langle \phi ,Z - W \rangle\]
for any $\phi \in \pi_2^0(\x,\y)$.

\textbf{Maslov grading}.  The relative Maslov grading of a pair $\x,\y$ is determined by the formula
\[M(\x) - M(\y) = \mu(\phi) - 2 n_w(\phi)\]
The Maslov index of the Whitney disk $\phi$ can be computed according to Lipshitz's formula \cite{Li}
\[\mu(\phi) = e(\phi) + \mu_{\x}(\phi) + \mu_{\y}(\phi)\]
where $e(\phi)$ is the Euler measure of the domain $\phi$ and $\mu_{\x},\mu_{\y}$ are point measures.  The {\it Euler measure} of a domain $\phi$ is $\frac{1}{2 \pi} \int_{\phi} \omega_g$ where $\omega_g$ is the curvature of a metric $g$ on $\Sigma$ for which all $\alpha,\beta$ curves are geodesics and always intersect at right angles.  The Euler measure is clearly additive and the Euler measure of a $2n$--gon is $1 - \frac{n}{2}$.  The {\it point measure} $\mu_x(\phi)$ of a domain at a vertex $x$ is the average of the coefficients of $\phi$ for the four elementary regions incident to the vertex $x$.  The point measure $\mu_{\x}(\phi) = \sum \mu_{x_i}(\phi)$ with respect to a generator is the sum of point measures of $\phi$ at each vertex $x_i$ of the generator.

Define $E \in R$ to be the vector whose $r^{\text{th}}$ entry is the Euler measure $e(r)$.  Then the relative Maslov grading is given by
\[M(\x) - M(\y) = \langle \phi, E \rangle + \langle |\mathcal{D}|(\phi), x+ y \rangle - 2 \langle \phi, W \rangle \]
for any $\phi \in \pi_2^0(\x,\y)$.

\begin{lemma}
\label{lemma:gradings_complexity}
The relative Alexander and Maslov gradings can be computed from $\SH_w^{nice}$ in $O(v^3 + v^2n + n^2 + (v^2+n)\left(\frac{v}{n}\right)^n)$.
\end{lemma}

\begin{proof}
Gauss-Jordan elimination on a $k \times l$ matrix takes $O(k^2 + l^2 k + l^3)$ and so computing the pseudoinverse $\cW$ takes $O(v^3 + v^2n + n^2)$.  Obtaining $\cW(x-y)$ takes $O(vn + n^2)$ and checking all $\cW(x-y) \pm \frac{1}{2} P_{i,j}$ to see if they are domains takes a further $O(vn)$.

To compute the inner product $\langle \phi,Z-W \rangle$ takes $O(n)$ time since $Z-W$ has at most $2n$ nonzero entries.  Thus it takes $O(n \left(\frac{v}{n}\right)^n)$ to compute all Alexander gradings.

Computing $\langle \phi, E \rangle$ takes $O(v)$, computing $|\cD|(\phi)$ via matrix multiplication takes $O(v^2)$ and computing $\langle |\mathcal{D}|(\phi), x+ y \rangle$ and $\langle \phi, W \rangle$ each take $O(n)$.  Thus computing all Maslov gradings takes $O((v^2+n)\left(\frac{v}{n}\right)^n)$.
\end{proof}

\textbf{Differential}.  Given a nice diagram, the Sarkar-Wang result \cite{SW} guarantees a unique holomorphic representative for each empty, embedded square or bigon on the Heegaard diagram connecting two generators.  Moreover, these are the only possible holomorphic disks and so the differential can be computed in terms of these squares or bigons.

Thus, to compute the differential we only need to consider pairs of generators $\x,\y$ that differ in at most two vertices.  Specifically,
\begin{align*}
\x = (x_1,x_2, x_3,\dots,x_{n-1+g}) \qquad \text{and} \qquad \y &= (y_1, y_2, x_3, \dots, x_{n-1+g}) \qquad \text{or} \\
& = (y_1,x_2,x_3,\dots,x_{n-1+g})
\end{align*}

In the first case, no square region contains a $\w$ basepoint, so the relative gradings satisfy $A(\x) = A(\y)$ and $M(\x) = M(\y) + 1$ and there is no ambiguity in $\mathcal{W}(x-y)$.  A Whitney disk $\phi \in \pi_2^0(\x,\y)$ is unique and if it is embedded then it is empty.  The $U$ coefficients always vanish.

In the second case, there can be at most two disks $\phi \in \pi_2^0(\x,\y)$ connecting $\x$ to $\y$.  Again, if $\x,\y$ have the appropriate relative gradings then an embedded disk is empty.  The $U$ coefficients can be computed by taking the inner product $\langle \phi, W \rangle$ which takes $O(n)$ time.

\begin{lemma}
\label{lemma:differential_complexity}
The differential can be computed in $O(v \left(\frac{v}{n}\right)^{2n})$.
\end{lemma}

\begin{proof}
It takes $O(v)$ to check whether a domain between a given pair of generators is embedded.
\end{proof}

\textbf{Homology}.  Having computed the relative gradings for all generators as well as the differential $\del^-$, we obtain the chain complex $\CFKm$ and can compute the homology.

Tabulating the combined time required to perform each of the above steps, we have the following proposition:

\begin{proposition}
\label{prop:homology_complexity}
The homology $\HFKm$ can be computed from a diagram $\SH_w^{nice}$ of complexity $(n,v)$ in $O( \left(\frac{v}{n}\right)^{3n})$ time.
\end{proposition}

\begin{proof}
Given the differential $\del^-$, the homology $\HFKm$ can be computed via Gauss-Jordan elimination in $O(\left(\frac{v}{n}\right)^{3n})$.  The proposition now follows by combining this fact with Lemmata \ref{lemma:gradings_complexity} and \ref{lemma:differential_complexity} and the fact that $n < v \leq \left(\frac{v}{n}\right)^n$
\end{proof}

Theorem~\ref{thm:alg_time} now follows by combining Propositions \ref{prop:diagram_complexity} and \ref{prop:homology_complexity}.


\section{Example} 
\label{sec:example}

As an example, we compare the computational complexity of this braid algorithm to the GRID algorithm for torus knots.  The actual homology is known and is fairly straightforward to obtain.  However, this class provides a good comparison to understand the algorithmic strengths of this braid approach.

Let $T(p,q)$ denote a torus knot, with $q > 0$.  Then it follows from work of Etnyre and Honda \cite{EH} that the arc index, and equivalently the minimal grid size, of $T(p,q)$ is $|p|+q$.  Using grid diagrams will therefore take $O(((p+q)!)^3)$ to compute the knot Floer homology.

The knot can be represented as the closure of the $q$-strand braid
\[B_{p.q} = (\sigma_1 \sigma_2 \cdots \sigma_{q-1})^p \]
which has length $(q-1)^p$ in the Artin letters.

The upper bound in Lemma \ref{lemma:complexity_runtime} significantly overestimates the time required to obtain the Heegaard diagram.

From now on assume that $p > 0$; the case of negative $p$ follows similarly.  Since each twist is positive, the number of $\alpha,\beta$ intersections and the number of generators of the chain complex grow monotonically.  Note that if $p = kq$ for some integer $k$, then the braid is just $k$ full twists.  Hence, we can bound the complexity of the diagram and time complexity of the algorithm by studying full twists.

A full twist is equivalent to the following.  Separate the $q$ strands into two collections: the first $i$ strands and the final $q-i$ strands.  Perform a full twist on the first $i$ strands then perform a full twist on the final $q- i$ strands.  Finally, perform a full twist exchanging the two collections of strands.

Since the Heegaard diagram depends only on the braid, this makes it easy to see what happens to $\beta_i$.

Twists among the first $i$ strands and final $q-i-1$ strands do not affect $\beta_i$, but the full twist of the two groups of strands does. Each half twist draws $\beta_i$ across each $\alpha$ curve exactly twice and so a full twist introduces 4 new intersections between $\beta_i$ and each $\alpha$ curve.  This is true for all $i$ and so each full twist adds $4(q-i)^2$ new vertices to the diagram.

If $p$ is not a multiple of $q$, then there will be some hexagonal regions that will need to be stabilized away, adding at most $q-2$ new pairs of $\alpha,\beta$ curves and at most sextupling the number of vertices.

Therefore, it follows that the complexity of the Heegaard diagram $\SH_{T(p,q)}^{nice}$ is at most $(2q-3, 24(\frac{p}{q}+1) (q-1)^2)$ and it can be obtained in $O(pq+q^2)$ time.

We can also bound the number of generators fairly easily by using the fact that the permanent is multilinear and that each quantity $|\alpha_i \cap \beta_j|$ grows linearly in $\frac{p}{q}$.  Thus the number of generators of the knot Floer complexes obtained from $\SH_{T(p,q)}^{nice}$ is bounded by 
\[\left(4 \frac{p}{q}\right)^{q-1} \left(8 \frac{p}{q}\right)^{q-2} (2q-3)! \leq c^q p^{2q-3}\]
for some constant $c$ not depending on $p$ or $q$.

Combining these facts with Lemmata \ref{lemma:gradings_complexity} and \ref{lemma:differential_complexity} and applying Gauss-Jordan elimination shows that $\HFKm(T(p,q))$ can be computed in $O((cp^2)^{3q} q^4)$ time.

Thus, for fixed $q$, the computational complexity grows polynomially in $p$, with degree determined by $q$.  For $p \gg q$ the braid approach becomes significantly faster.  However, for $p \sim q$, the grid diagram approach is faster as $(2q)! \leq q^{2q}$.


\bibliographystyle{myalpha}
\nocite{*}
\bibliography{References}

\end{document}